\documentclass[11pt]{article}
\usepackage{amssymb}
\newtheorem{precor}{{\bf Corollary}}

\newtheorem{precon}{{\bf Conjecture}}

\newtheorem{prealphcon}{{\bf Conjecture}}

\newtheorem{predefin}{{\bf Definition}}

\newtheorem{preexm}{{\bf Example}}

\newtheorem{preappl}{{\bf Application}}

\newtheorem{prelem}{{\bf Lemma}}

\newenvironment{lem}{\begin{prelem}{\hspace{-0.5
               em}{\bf.\ }}}{\end{prelem}}
\newtheorem{preproof}{{\bf Proof.\ }}

\newenvironment{proof}[1]{\begin{preproof}{\rm
               #1}\hfill{$\blacksquare$}}{\end{preproof}}
\newtheorem{pretheorem}{{\bf Theorem}}

\newenvironment{theorem}{\begin{pretheorem}{\hspace{-0.5
               em}{\bf.\ }}}{\end{pretheorem}}
\newtheorem{prealphtheorem}{{\bf Theorem}}

\newtheorem{prealphlem}{{\bf Lemma}}

\newtheorem{prepro}{{\bf Proposition}}

\newtheorem{preprb}{{\bf Problem}}

\newtheorem{prerem}{{\bf Remark}}

\newtheorem{preapp}{{\bf Application}}

\newtheorem{prequ}{{\bf Question}}

\def\conct[#1,#2]{\mbox {${#1} \leftrightarrow {#2}$}}
\def\dconct[#1,#2]{\mbox {${#1} \rightarrow {#2}$}}
\def\deg[#1,#2]{\mbox {$d_{_{#1}}(#2)$}}
\def\mindeg[#1]{\mbox {$\delta_{_{#1}}$}}
\def\maxdeg[#1]{\mbox {$\Delta_{_{#1}}$}}
\def\outdeg[#1,#2]{\mbox {$d_{_{#1}}^{^+}(#2)$}}
\def\minoutdeg[#1]{\mbox {$\delta_{_{#1}}^{^+}$}}
\def\maxoutdeg[#1]{\mbox {$\Delta_{_{#1}}^{^+}$}}
\def\indeg[#1,#2]{\mbox {$d_{_{#1}}^{^-}(#2)$}}
\def\minindeg[#1]{\mbox {$\delta_{_{#1}}^{^-}$}}
\def\maxindeg[#1]{\mbox {$\Delta_{_{#1}}^{^-}$}}

\def\dre[#1,#2,#3]{\mbox {${\cal E}^{^{#3}}(#1,#2)$}}
\def\var[#1,#2]{\mbox {${\rm Var}_{_{#1}}(#2)$}}
\def\ls[#1]{\mbox {$\xi^{^{#1}}$}}
\def\hom[#1,#2]{\mbox {${\rm Hom}({#1},{#2})$}}
\def\onvhom[#1,#2]{\mbox {${\rm Hom^{v}}(#1,#2)$}}
\def\onehom[#1,#2]{\mbox {${\rm Hom^{e}}(#1,#2)$}}
\def\core[#1]{\mbox {$#1^{^{\bullet}}$}}
\def\cay[#1,#2]{\mbox {${\rm Cay}({#1},{#2})$}}
\def\sch[#1,#2,#3]{\mbox {${\rm Sch}({#1},{#2},{#3})$}}
\def\cays[#1,#2]{\mbox {${\rm Cay_{s}}({#1},{#2})$}}
\def\dirc[#1]{\mbox {$\stackrel{\rightarrow}{C}_{_{#1}}$}}
\def\cycl[#1]{\mbox {${\bf Z}_{_{#1}}$}}

\begin{document}

\begin{center}
{\Large \bf The b-Chromatic Number of Regular Graphs via The Edge Connectivity}\\
\vspace{0.3 cm}
{\bf Saeed Shaebani}\\
{\it Department of Mathematical Sciences}\\
{\it Institute for Advanced Studies in Basic Sciences {\rm (}IASBS{\rm )}}\\
{\it P.O. Box {\rm 45195-1159}, Zanjan, Iran}\\
{\tt s\_shaebani@iasbs.ac.ir}\\ \ \\
\end{center}
\begin{abstract}
\noindent The b-chromatic number of a graph $G$, denoted by
$\varphi(G)$, is the largest integer $k$ that $G$ admits a proper
coloring by $k$ colors, such that each color class has a vertex
that is adjacent to at least one vertex in each of the other
color classes. El Sahili and Kouider [About b-colorings of regular
graphs, Res. Rep. $1432$, LRI, Univ. Orsay, France, 2006] asked
whether it is true that every $d$-regular graph $G$ of girth at
least $5$ satisfies $\varphi(G)=d+1$. Blidia, Maffray, and Zemir
[On b-colorings in regular graphs, Discrete Appl. Math. 157
(2009), 1787-1793] showed that the Petersen graph provides a
negative answer to this question, and then conjectured that the
Petersen graph is the only exception. In this paper, we
investigate a strengthened form of the question.

The edge connectivity of a graph $G$, denoted by $\lambda(G)$, is
the minimum cardinality of a subset $U$ of $E(G)$ such that
$G\setminus U$ is either disconnected or a graph with only one
vertex. A $d$-regular graph $G$ is called super-edge-connected if
every minimum edge-cut is the set of all edges incident with a
vertex in $G$, i.e., $\lambda(G)=d$ and every minimum edge-cut of
$G$ isolates a vertex. We show that if $G$ is a $d$-regular graph
that contains no $4$-cycle, then $\varphi(G)=d+1$ whenever $G$ is
not super-edge-connected.
\\

\end{abstract}
\section{Introduction}

All graphs considered in this paper are finite and simple
(undirected, loopless, and without multiple edges). Let $G=(V,E)$
be a graph. A  ({\it proper vertex}) {\it coloring} of $G$, is a
function $f_{G}:V(G)\rightarrow C$ such that for each $\{u,v\}$
in $E(G)$, $f_{G}(u)\neq f_{G}(v)$. Each $c$ in $C$ is called a
color. Also, for each $c$ in $C$, $f_{G}^{-1}(c)$ is called a
color class of $f_{G}$. We say $v$ is colored by $c$ if
$f_{G}(v)=c$. We mean by $\chi (G)$, the minimum cardinality of a
set $C$ that a coloring $f_{G}:V(G)\rightarrow C$ exists. A
$b$-{\it coloring} of the graph $G$ is a coloring
$f_{G}:V(G)\rightarrow C$ such that for each $c$ in $C$, there
exists some vertex $v$ in $V(G)$ such that $f_{G}(v)=c$ and
$f_{G}(N_{G}(v))=C\setminus \{c\}$, where $N_{G}(v):=\{w|\
\{v,w\}\in E(G)\}$. In other words, a coloring of $G$ is called a
b-coloring, if each color class contains a vertex that is
adjacent to at least one vertex in each of the other color
classes. Obviously, each coloring of $G$ with $\chi (G)$ colors,
is a b-coloring. Also, if $f_{G}:V(G)\rightarrow C$ is a
b-coloring, then $|C| \leq \Delta (G) + 1$, where $\Delta (G)$
denotes the maximum degree of $G$. The {\it b-chromatic number}
of $G$, denoted by $\varphi(G)$, is the maximum cardinality of a
set $C$ such that a b-coloring $f_{G}:V(G)\rightarrow C$ exists.
The concept of b-coloring of graphs introduced by Irving and
Manlove in 1999 in \cite{irv}, and has received attention.

Given a graph $G$ and a coloring $f_{G}:V(G)\rightarrow C$,
a vertex $v$ in $V(G)$ is called a {\it color-dominating}
vertex (with respect to $f_{G}$) if $f_{G}(N_{G}(v))=C\setminus \{f_{G}(v)\}$.
In other words, $v$ is a color-dominating vertex, if it is
adjacent to at least one vertex in each of the other color
classes. We say that the color $c$ {\it realizes}, if there
exists a color-dominating vertex which is colored by $c$.

As mentioned, for each graph $G$ with maximum degree $\Delta(G)$,
$\varphi(G)\leq\Delta(G)+1$. Therefore, for each $d$-regular
graph $G$, $\varphi(G)\leq d+1$. Since $d+1$ is the maximum
possible b-chromatic number for $d$-regular graphs, determining
necessary or sufficient conditions to achieve this bound is of
interest. Kratochvil, Tuza, and Voigt \cite{kratv} proved that
every $d$-regular graph $G$ with at least $d^{4}$ vertices
satisfies $\varphi(G)=d+1$. In \cite{ca.ja}, Cabello and Jakovac
lowered $d^{4}$ to $2d^{3}-d^{2}+d$. These amazing bounds confirm
that for each natural number $d$, there are only a finite number
of $d$-regular graphs (up to isomorphism) that their b-chromatic
numbers are not $d+1$. El Sahili and Kouider \cite{sa.ko} asked
whether it is true that every $d$-regular graph $G$ of girth at
least $5$ satisfies $\varphi(G)=d+1$. In this regard, Blidia,
Maffray, and Zemir \cite{b.m.z} showed that the Petersen graph
provides a negative answer to this question. They proved that the
b-chromatic number of the Petersen graph is 3, and then
conjectured that the Petersen graph is the only exception. They
also proved this conjecture for $d\leq6$. In \cite{kou1}, Kouider
proved that the b-chromatic number of any $d$-regular graph of
girth at least 6 is $d+1$. El Sahili and Kouider \cite{sa.ko}
showed that the b-chromatic number of any $d$-regular graph of
girth $5$ that contains no $6$-cycle is $d+1$. In \cite{ca.ja},
Cabello and Jakovac proved a celebrated theorem for the
b-chromatic number of regular graphs of girth $5$, which
guarantees that the b-chromatic number of each $d$-regular graph
with girth $5$, is bounded below by a linear function of $d$.
They proved that a $d$-regular graph with girth at least $5$, has
b-chromatic number at least $\lfloor\frac{d+1}{2}\rfloor$. Also,
they proved that for except small values of $d$, every connected
$d$-regular graph that contains no $4$-cycle and its diameter is
at least $d$, has b-chromatic number $d+1$. It is shown in
\cite{saeed2} that if $G$ is a $d$-regular graph that contains no
$4$-cycle, then $\varphi(G)\geq\lfloor\frac{d+3}{2}\rfloor$. This
lower bound, is sharp for the Petersen graph. Besides, If $G$ has
a triangle, then $\varphi(G)\geq\lfloor\frac{d+4}{2}\rfloor$.
Also, if $G$ is a $d$-regular graph that contains no $4$-cycle and
$diam(G)\geq6$, then $\varphi(G)=d+1$.

The {\it vertex connectivity} of a graph $G$, denoted by
$\kappa(G)$, is the minimum cardinality of a subset $U$ of $V(G)$
such that $G\setminus U$ is either disconnected or a graph with
only one vertex. Also, the {\it edge connectivity} of a graph $G$,
denoted by $\lambda(G)$, is the minimum cardinality of a subset
$U$ of $E(G)$ such that $G\setminus U$ is either disconnected or
a graph with only one vertex. It is well-known that for each
graph $G$, $\kappa(G)\leq \lambda(G)\leq \delta(G)$, where
$\delta(G)$ denotes the minimum degree of $G$. By an {\it
edge-cut} of $G$, we mean a subset of $E(G)$ such that deleting
all of its elements from $G$, yields a disconnected graph.
Therefore, for each graph $G$ with at least two vertices,
$\lambda(G)$ is the minimum cardinality of all edge-cuts of $G$.
For every edge-cut $T$ of $G$, $sat(T)$ stands for the set of all
vertices of $G$ that are saturated by some edges in $T$, i.e.,
$sat(T):=\{ v |\ v\in V(G),$ there exists some $w$ in $V(G)$ for
which $\{v,w\}\in T \}$. We mean by a minimum edge-cut of $G$, an
edge-cut of $G$ with cardinality $\lambda(G)$. A $d$-regular
graph $G$, is called {\it super-edge-connected}, if every minimum
edge-cut of $G$, is the set of all edges incident with a vertex
in $G$, i.e., $\lambda(G)=d$ and deleting each minimum edge-cut
of $G$ from $G$, yields a graph which has an isolated vertex.

An edge-cut of a graph $G$ is called {\it trivial} whenever it is
equal to the set of all edges incident with a vertex of $G$. With
this terminology, a $d$-regular graph $G$ is super-edge-connected
if and only if every minimum edge-cut of $G$ is trivial.

It has been proved in \cite{saeed2} that for any $d$-regular
graph $G$ that contains no $4$-cycle, if $\kappa(G)\leq
\frac{d+1}{2}$, then $\varphi(G)=d+1$. This upper bound is sharp
in the sense that the vertex connectivity of the Petersen graph
is $\frac{d+1}{2}+1$; nevertheless, its b-chromatic number is not
$d+1$. Also, if $\kappa(G)<\frac{3d-3}{4}$, then
$\min\{2(d-\kappa(G)+1),d+1\}\leq\varphi(G)\leq d+1$.
Furthermore, if there exists a subset $U$ of $V(G)$ such that
$|U|=\kappa(G)$ and $G\setminus U$ has at least four connected
components, then $\varphi(G)=d+1$. Moreover, if
$\kappa(G)<\frac{2d-1}{3}$ and there exists a subset $U$ of
$V(G)$ such that $|U|=\kappa(G)$ and $G\setminus U$ has at least
three connected components, then $\varphi(G)=d+1$. If the girth
of $G$ is $5$, $\frac{3d-3}{4}$ and $\frac{2d-1}{3}$ can be
replaced by $\frac{3d}{4}$ and $\frac{2d+1}{3}$, respectively.

In this paper, we investigate the b-chromatic number of
$d$-regular graphs with no $4$-cycles. We show that if $G$ is a
$d$-regular graph that contains no $4$-cycle, then
$\varphi(G)=d+1$ whenever $G$ is not super-edge-connected.
Throughout the paper, for each nonnegative integer $n$, the
symbol $[n]$ stands for the set $\{i|\ i\in \mathbb{N},\ 1\leq
i\leq n \}$.

\section{ The Main Result}

This section concerns a relation between the b-chromatic number
and the edge-connectivity in $d$-regular graphs that do not
contain $4$-cycles. We show that every $d$-regular graph that
does not contain $C_{4}$ as a subgraph, achieves the maximum
b-chromatic number $d+1$, unless it is super-edge-connected. In
this regard, first we mention Lemma \ref{majorlemma} and Lemma
\ref{mainlemma}; the former is related to the super-edge-connected
graphs without $4$-cycle, and the latter presents a sufficient
condition for bipartite graphs to have a perfect matching.

\begin{lem} \label{majorlemma}
{Let $d\geq 4$ and $G$ be a $d$-regular graph that contains no
$4$-cycle, and $T$ be a minimum edge-cut of $G$ which is not
trivial. Suppose that $G_{1},\ldots,G_{l}$ are connected
components of $G\setminus T$. Then for each $i$ in $[l]$, there
exists some $a_{i}$ in $V(G_{i}) \setminus sat(T)$ such that
$|N_{G}(a_{i})\bigcap sat(T)|\leq 2$. }
\end{lem}

\begin{proof}
{ Let us regard an arbitrary $i$ in $[l]$ as fixed. Set
$A_{i}:=V(G_{i}) \bigcap sat(T)$ and $s:=|V(G_{i})|$. Since $T$ is
not trivial, $s\geq 2$. Obviously, $s\neq 2$; otherwise, the
number of edges between $V(G_{i})$ and $V(G) \setminus V(G_{i})$
in the graph $G$ is at least $2(d-1)$. So $|T| \geq 2(d-1) \geq
d+2$, a contradiction. Hence, $s\geq 3$. It is well-known that
the number of edges of a graph with $n$ vertices which contains
no $4$-cycle is at most $\frac{n}{4} (1+\sqrt{4n-3})$. Since the
graph $G_{i}$ does not contain any $4$-cycles, $sd=\Sigma_{x\in
V(G_{i})}deg_{G}(x)\leq \frac{s}{2} (1+\sqrt{4s-3})+|T|\leq
\frac{s}{2} (1+\sqrt{4s-3})+d$. Hence,

$$
\begin{array}{cclc}
  (s-1)d\leq \frac{s}{2} (1+\sqrt{4s-3}) & \Longrightarrow & 2(s-1)d-s\leq s\sqrt{4s-3} &  \\
  & \Longrightarrow & (2d-1)(s-1)-1\leq s\sqrt{4s-3} &  \\
  & \Longrightarrow & (2d-2)(s-1) < s\sqrt{4s-3}  & ({\rm since}\ 2<s) \\
  & \Longrightarrow & d-1 < \frac{s}{2(s-1)} \sqrt{4s-3}. &
\end{array}
$$

One can easily observe that the derivative of the function
$f(x)=\frac{x}{2(x-1)} \sqrt{4x-3}-(d-1)$ is positive for $x\in
[3,+\infty)$. So $f$ is strictly increasing in the interval
$[3,+\infty)$. Thus, proving $f(d+3) < 0$, implies $s \geq d+4$.
Since the derivative of the function $g(y)=\frac{y+3}{2(y+2)}
\sqrt{4y+9}-(y-1)$ is negative for $y\in (0,+\infty)$ and
$g(4)=\frac{-1}{3} < 0$, we obtain that for each natural number
$d$ which $d \geq 4$, $g(d) < 0$. Hence, $f(d+3)=g(d) < 0$; and
therefore, $s \geq d+4$. Since $|A_{i}| \leq |T| \leq d $,
$|V(G_{i})\setminus A_{i}| \geq 4$. Now, consider four elements
$x_{1}$, $x_{2}$, $x_{3}$, and $x_{4}$, in $V(G_{i})\setminus
A_{i}$. Since $G$ does not have any $4$-cycles, there exists some
$j$ in $\{1,2,3,4\}$ such that $|N_{G}(x_{j})\bigcap A_{i}| < d$.
So $|N_{G}(x_{j}) \setminus A_{i}| > 0$. Set $N_{G}(x_{j})
\setminus A_{i}=\{y_{k} | 1\leq k \leq |N_{G}(x_{j}) \setminus
A_{i}|\}$. Obviously, there exists some $k$ in $[|N_{G}(x_{j})
\setminus A_{i}|]$ such that $|N_{G}(y_{k}) \bigcap
(A_{i}\setminus N_{G}(x_{j}))| \leq 1$; otherwise, $d\geq |A_{i}|
\geq |N_{G}(x_{j})\bigcap A_{i}| + 2|N_{G}(x_{j}) \setminus
A_{i}| = d + |N_{G}(x_{j}) \setminus A_{i}|
> d$, which is impossible. We conclude that there exists an
element $y_{k}$ in $N_{G}(x_{j}) \setminus A_{i}$ for which
$|N_{G}(y_{k}) \bigcap (A_{i}\setminus N_{G}(x_{j}))| \leq 1$.
Since $y_{k}$ has at most one neighbor in $N_{G}(x_{j})\bigcap
A_{i}$; hence, $y_{k}$ has at most two neighbors in $A_{i}$.
Therefore, $a_{i}:=y_{k}$ is a desired vertex.

}
\end{proof}

\begin{lem} \label{mainlemma} {\rm \cite{ca.ja}}
{Let $H$ be a bipartite graph with parts $U$ and $V$ such that
$|U|=|V|$. Let $u^{*}\in U$ and $v^{*}\in V$. If for each vertex
$x$ in $V(H)\setminus\{u^{*},v^{*}\}$,
$deg_{H}(x)\geq\frac{|V|}{2}$, $deg_{H}(u^{*})>0$, and
$deg_{H}(v^{*})>0$, then $H$ has a perfect matching. }
\end{lem}

We are now in a position to prove the main result of the paper.

\begin{theorem}{ Let $G$ be a $d$-regular graph that contains
no $4$-cycle. If $G$ is not super-edge-connected, then
$\varphi(G)=d+1$.
 }
\end{theorem}

\begin{proof}
{There is nothing to prove when $d\in \{0,1,2\}$. Also, Jakovac
and Klav\v{z}ar, in \cite{ja.kl}, showed that the only cubic graph
that contains no $4$-cycle and its b-chromatic number is not
equal to $4$, is the Petersen graph. Since the Petersen graph is
super-edge-connected, the proof is completed for $d=3$. So we
suppose that $d\geq 4$. Since $G$ is not super-edge-connected,
there exists a minimum edge-cut $T$ of $G$ which is not trivial.
Suppose that $G_{1},\ldots,G_{l}$ are connected components of
$G\setminus T$. For each $i$ in $[l]$, define $A_{i}:=\{x|\ x\in
V(G_{i}),\ \exists y\in V(G)\setminus V(G_{i})\ {\rm such\ that}
\ \{x,y\}\in T \}$. According to the Lemma \ref{majorlemma}, for
each $i$ in $[l]$, there exists some $a_{i}$ in $V(G_{i})
\setminus A_{i}$ for which $|N_{G}(a_{i})\bigcap A_{i}|\leq 2$.

Let $x_{1},\ldots,x_{d-|N_{G}(a_{1})\bigcap A_{1}|}$ be an
arbitrary ordering of all elements of $N_{G}(a_{1})\setminus
A_{1}$. Color the vertex $a_{1}$ by color $1$, and for each $i$ in
$[|N_{G}(a_{1})\setminus A_{1}|]$, assign the color $i+1$ to the
vertex $x_{i}$. Also, color all vertices that are in the set
$N_{G}(a_{1})\bigcap A_{1}$ by all colors that are in the set
$[d+1] \setminus [|N_{G}(a_{1})\setminus A_{1}|+1]$ injectively.
Then, color all vertices that are in the set $N_{G}(a_{2})\bigcap
A_{2}$ by some colors that are in the set $\{1, \lfloor
\frac{d+2}{2} \rfloor \}$ injectively.

The vertex $a_{1}$ is a color-dominating vertex with color $1$.
Now, our task is to color all the vertices in $(\bigcup
_{i=1}^{\lfloor \frac{d}{2} \rfloor}N_{G}(x_{i})) \setminus
(\{a_{1}\} \bigcup N_{G}(a_{1}))$ by colors in $[d+1]$ such that
for each $i$ in $[\frac{d}{2}]$, all colors that are in the set
$[d+1] \setminus \{i+1\}$, appear on $N_{G}(x_{i})$.

For each $i$ in $[\lfloor \frac{d}{2} \rfloor]$, set $V_{i}$,
$S_{i}$, and $C_{i}$, as follows:

$$
\begin{array}{rl}
\bullet\ V_{i}:= & N_{G}(x_{i})\setminus (\{a_{1}\}\bigcup N_{G}(a_{1})); \\
\bullet\ S_{i}:= & \{a_{1}\}\bigcup N_{G}(a_{1})\bigcup
(\bigcup_{j=1}^{i}V_{j})\bigcup (N_{G}(a_{2}) \bigcap
A_{2});  \\
\bullet\ C_{i}:=& ([d+1]\setminus \{1\} )\setminus (\mbox{the set
of colors
that were appeared on} \\
& \{x_{i}\}\bigcup(N_{G}(x_{i})\bigcap N_{G}(a_{1}))).
\end{array}$$

Since $G$ contains no $4$-cycle, for any two distinct natural
numbers $i$ and $j$ in $[\lfloor \frac{d}{2} \rfloor]$,
$V_{i}\bigcap V_{j}=\varnothing$. Also, since $G$ contains no
$4$-cycle, the maximum degree of the induced subgraph of $G$ on
$N_{G}(a_{1})$ is at most one. So $|V_{i}|=|C_{i}|=d-1$ or
$|V_{i}|=|C_{i}|=d-2$. Moreover, $|V_{i}|=|C_{i}|=d-2$ if and
only if $|N_{G}(x_{i})\bigcap N_{G}(a_{1})|=1$. Now, we follow
$\lfloor \frac{d}{2} \rfloor$ steps inductively. For each $i$ in
$[\lfloor \frac{d}{2} \rfloor]$, at $i$-th step, we only color
all vertices that are in $V_{i}$ by all colors that are in
$C_{i}$ injectively. Suppose by induction that $1\leq i\leq
\lfloor \frac{d}{2} \rfloor$ and for each $k$ in $[i-1]$, at
$k$-th step, we have only colored all vertices that are in
$V_{k}$ by all colors that are in $C_{k}$ injectively in such a
way that the resulting partial coloring on $S_{k}$ is a proper
coloring. Now, at $i$-th step, we want to color only all vertices
that are in $V_{i}$ by all colors that are in $C_{i}$ injectively
such that the resulting partial coloring on $S_{i}$ be a proper
partial coloring. Consider a bipartite graph $H_{i}$ with one part
$V_{i}$ and the other part $C_{i}$, which a vertex $v$ in $V_{i}$
is adjacent to a color $c$ in $C_{i}$ in the graph $H_{i}$ if and
only if (in the graph $G$) $v$ does not have any neighbors in
$S_{i}$ already colored by $c$. Such a coloring of all vertices
that are in $V_{i}$ by all colors that are in $C_{i}$ (as
mentioned) exists if and only if $H_{i}$ has a perfect matching.

Let $v$ be an arbitrary element of $V_{i}$. The set of neighbors
of $v$ in the graph $G$ that were already colored, is a subset of
$\{x_{i}\}\bigcup(\bigcup_{j=1}^{i-1}V_{j})\bigcup (N_{G}(a_{2})
\bigcap A_{2})$. Since $G$ contains no $4$-cycle, for each $j$ in
$[i-1]$, $v$ has at most one neighbor in $V_{j}$. Also, $v$ has
at most one neighbor in $N_{G}(a_{2}) \bigcap A_{2}$. However, the
color of the vertex $x_{i}$ does not belong to $C_{i}$.
Therefore, $deg_{H_{i}}(v)\geq |C_{i}|-i$. Also, since for each
$j$ in $[i-1]$, $v_{j}$ sees all colors of $[d+1]$ on its closed
neighborhood, each color of $[d+1]$ appears at most once on
$V_{j}$. Besides, each color of $[d+1]$ appears at most once on
$N_{G}(a_{2}) \bigcap A_{2}$. Therefore, for each $c$ in $C_{i}$,
$deg_{H_{i}}(c)\geq|V_{i}|-i$. Hence, for each $v$ in $V(H_{i})$,
$deg_{H_{i}}(v)\geq|V_{i}|-i$. Since $|V_{i}|\geq d-2$, if $1\leq
i\leq \lfloor \frac{d-2}{2} \rfloor$, then $deg_{H_{i}}(v)\geq
|V_{i}|-i\geq\frac{|V_{i}|}{2}$. In the case $i=\lfloor
\frac{d}{2} \rfloor$, since the color of each vertex in
$N_{G}(a_{2})\bigcap A_{2}$ belongs to the set $\{1, \lfloor
\frac{d+2}{2} \rfloor \}$ and also $C_{\lfloor \frac{d}{2}
\rfloor} \bigcap \{1, \lfloor \frac{d+2}{2} \rfloor \} =
\varnothing$,  for each $v$ in $V(H_{i})$,
$deg_{H_{i}}(v)\geq|V_{i}|-(i-1)\geq \frac{|V_{i}|}{2}$.

So Lemma \ref{mainlemma} implies that $H_{i}$ has a perfect
matching and we are done. We conclude that there exists a partial
coloring on $S_{\lfloor \frac{d}{2} \rfloor}$ by all colors of
the set $[d+1]$, such that $a_{1},x_{1},\ldots,x_{\lfloor
\frac{d}{2} \rfloor}$ are color-dominating vertices whose colors
are $1,2,\ldots,\lfloor \frac{d+2}{2} \rfloor$, respectively.

The next task is to color some uncolored vertices from $V(G_{2})$
in such a way that all colors in $[d+1]\setminus [\lfloor
\frac{d+2}{2} \rfloor]$ realize. The procedure is to find
$\lfloor \frac{d-1}{2} \rfloor$ suitable vertices
$z_{1},\ldots,z_{\lfloor \frac{d-2}{2} \rfloor}$ in $N_{G}(a_{2})
\setminus A_{2}$ in order to make them along $a_{2}$
color-dominating.

Let $N_{G}(a_{2}) \setminus A_{2}= \{y_{i}| 1\leq i \leq
|N_{G}(a_{2}) \setminus A_{2}| \}$. For each $i$ in
$[|N_{G}(a_{2}) \setminus A_{2}|]$, set $W_{i}$ and $e_{y_{i}}$
as follows:

\
\

\noindent $\bullet\ W_{i}:=N_{G}(y_{i})\setminus (\{a_{2}\}\bigcup
N_{G}(a_{2}))$;

\noindent $\bullet\ e_{y_{i}}:=|\{ \{s,t\} |\ \{s,t\} \in E(G),\
s\in W_{i},\ t\in A_{1}\}|$.

\
\

Since $G$ does not contain any $4$-cycles, for any two different
natural numbers $i$ and $j$ in $[|N_{G}(a_{2}) \setminus
A_{2}|]$, $W_{i}\bigcap W_{j}=\varnothing$. Without loss of
generality, we can assume that the sequence
$\{e_{y_{i}}\}_{i=1}^{|N_{G}(a_{2}) \setminus A_{2}|}$ is
decreasing, i.e., $e_{y_{1}} \geq \cdots \geq e_{y_{|N_{G}(a_{2})
\setminus A_{2}|}}$. We show that for each natural number $i$,
$|N_{G}(a_{2}) \setminus A_{2}|-\lfloor \frac{d-1}{2} \rfloor +1
\leq i \leq |N_{G}(a_{2}) \setminus A_{2}|$, $e_{y_{i}} \leq 1$.
Suppose, on the contrary, that for some natural number $i$,
$|N_{G}(a_{2}) \setminus A_{2}|-\lfloor \frac{d-1}{2} \rfloor +1
\leq i \leq |N_{G}(a_{2}) \setminus A_{2}|$, $e_{y_{i}}
> 1$. Therefore, for each $j$ in $[|N_{G}(a_{2}) \setminus
A_{2}|-\lfloor \frac{d-1}{2} \rfloor +1]$, $e_{y_{j}} \geq 2$.
Since each vertex in $N_{G}(a_{2}) \bigcap A_{2}$ is incident with
an edge of $T$, so

$$\begin{array}{rl}
  d\geq |T| \geq |N_{G}(a_{2}) \bigcap A_{2}| + 2 (|N_{G}(a_{2})
\setminus A_{2}|-\lfloor \frac{d-1}{2} \rfloor +1) &  =\\
|N_{G}(a_{2}) \bigcap A_{2}| + 2 (d-|N_{G}(a_{2}) \bigcap
A_{2}|- \lfloor \frac{d-3}{2} \rfloor )  &  \geq \\
 d+3-|N_{G}(a_{2}) \bigcap A_{2}| \geq d+3-2  & = d+1,
\end{array}$$

\noindent which is impossible. Accordingly, for each natural
number $i$, $|N_{G}(a_{2}) \setminus A_{2}|-\lfloor \frac{d-1}{2}
\rfloor +1 \leq i \leq |N_{G}(a_{2}) \setminus A_{2}|$,
$e_{y_{i}} \leq 1$. For each $i$ in $[\lfloor \frac{d-1}{2}
\rfloor]$, put $z_{i}:=y_{|N_{G}(a_{2}) \setminus A_{2}|-\lfloor
\frac{d-1}{2} \rfloor+i}$; therefore, $e_{z_{i}} \leq 1$.

Now, color the vertex $a_{2}$ by color $d+1$ and for each $i,\ 1
\leq i \leq \lfloor \frac{d-1}{2} \rfloor$, assign the color
$\lfloor \frac{d+2}{2} \rfloor + i $ to the vertex $z_{i}$. Also,
color all vertices that are in the set $N_{G}(a_{2}) \setminus (
A_{2} \bigcup \{z_{i}| 1 \leq i \leq  \lfloor \frac{d-1}{2}
\rfloor\} )$ by some colors of $[d+1]$ injectively in such a way
that all colors of the set $[d+1]$ appear on the closed
neighborhood of $a_{2}$.

For each $i$ in $[\lfloor \frac{d-1}{2} \rfloor]$, define
$V'_{i}$, $S'_{i}$, and $C'_{i}$, as follows:

\
\

\noindent $\bullet\ V'_{i}:=N_{G}(z_{i})\setminus
(\{a_{2}\}\bigcup N_{G}(a_{2}))$;

\noindent $\bullet\ S'_{i}:= \{a_{2}\}\bigcup N_{G}(a_{2})\bigcup$
$(\bigcup_{j=1}^{i}V'_{j}) \bigcup S_{\lfloor \frac{d}{2} \rfloor}
$;

\noindent $\bullet\ C'_{i}:=[d] \setminus ($the set of colors that
were appeared on $\{z_{i}\}\bigcup(N_{G}(z_{i})\bigcap
N_{G}(a_{2})))$.

\
\

The maximum degree of the induced subgraph of $G$ on
$N_{G}(a_{2})$ is at most one. So $|V'_{i}|=|C'_{i}|=d-1$ or
$|V'_{i}|=|C'_{i}|=d-2$. Furthermore, $|V'_{i}|=|C'_{i}|=d-2$ if
and only if $|N_{G}(z_{i})\bigcap N_{G}(a_{2})|=1$. Now, we follow
$\lfloor \frac{d-1}{2} \rfloor$ steps inductively. For each $i$ in
$[\lfloor \frac{d-1}{2} \rfloor]$, at $i$-th step, we only color
all vertices that are in $V'_{i}$ by all colors that are in
$C'_{i}$ injectively. Suppose by induction that $1\leq i\leq
\lfloor \frac{d-1}{2} \rfloor$ and for each $k$ in $[i-1]$, at
$k$-th step, we have only colored all vertices that are in
$V'_{k}$ by all colors in $C'_{k}$ injectively in such a way that
the resulting partial coloring on $S'_{k}$ is a proper coloring.
Now, at $i$-th step, we aim to color only all vertices that are in
$V'_{i}$ by all colors that are in $C'_{i}$ injectively such that
the resulting partial coloring on $S'_{i}$ be a proper partial
coloring. Consider a bipartite graph $H'_{i}$ with one part
$V'_{i}$ and the other part $C'_{i}$, that a vertex $v$ in
$V'_{i}$ is adjacent to a color $c$ in $C'_{i}$ in the graph
$H'_{i}$, if and only if (in the graph $G$) $v$ does not have any
neighbors in $S'_{i}$ already colored by $c$. Such a coloring of
all vertices that are in $V'_{i}$ by all colors that are in
$C'_{i}$ exists if and only if $H'_{i}$ has a perfect matching.
So our goal is to prove the existence of a perfect matching in
$H'_{i}$. In this regard, we again apply Lemma \ref{mainlemma}.

Since $e_{z_{i}} \leq 1$, there exists at most one edge between
$V'_{i}$ and $A_{1}$ in the graph $G$. Thus, there exists at most
one element in $V'_{i}$ that has a neighbors in $V(G_{1})$ (in the
graph $G$). If such a vertex exists, call it $v^{*}$. Therefore,
in the graph $H'_{i}$, the degree of each vertex in $V'_{i}
\setminus \{v^{*}\}$ is at least $|C'_{i}|-(i-1)$. Similarly,
there exists at most one color in $C'_{i}$ for which there is an
element in $V'_{i}$ that has a neighbor in $V(G_{1})$ with this
color. If such a color exists, call it $c^{*}$. Hence, in the
graph $H'_{i}$, the degree of each vertex in $C'_{i} \setminus
\{c^{*}\}$ is at least $|V'_{i}|-(i-1)$. We conclude that for each
vertex $x$ in $V_{H'_{i}} \setminus \{v^{*},c^{*}\}$,
$deg_{H'_{i}}(x)\geq |V'_{i}|-(i-1)$; and since $i \leq \lfloor
\frac{d-1}{2} \rfloor$ and $|V'_{i}|\geq d-2$, $deg_{H'_{i}}(x)
\geq \frac{|V'_{i}|}{2}$. Also, the degree of each vertex of
$H'_{i}$ is at least $|V'_{i}|-i$ which is positive.

We conclude that $|V'_{i}|=|C'_{i}|$ and the degree of each vertex
in $V_{H'_{i}} \setminus \{v^{*},c^{*}\}$ is at least
$\frac{|V'_{i}|}{2}$. Also, the degree of each vertex of $H'_{i}$
is positive. Accordingly, the Lemma~\ref{mainlemma} implies that
$H'_{i}$ has a perfect matching. Therefore, there exists a partial
coloring on $S'_{\lfloor \frac{d-1}{2} \rfloor}$ by all colors of
the set $[d+1]$ such that all colors realize. This partial
coloring can be extended to a coloring of the graph $G$ greedily.
So $\varphi(G)=d+1$.

}
\end{proof}

\ \\
{\bf Acknowledgements:} I am grateful to Professor Hossein
Hajiabolhassan for his many invaluable comments. Also, I would
like to express my gratitude to Professor Rashid Zaare-Nahandi
for his support during my M.Sc. and Ph.D. education.

\end{document}